 \documentclass[12pt]{amsart}

\usepackage{amssymb, amsmath, amsfonts}
\usepackage[raiselinks=false,colorlinks=true,citecolor=blue,urlcolor=blue,
linkcolor=blue,bookmarksopen=true,pdftex]{hyperref}

\usepackage{enumerate}
\usepackage[mathscr]{eucal}
\newtheorem{theorem}{Theorem}[section]
\newtheorem{proposition}[theorem]{Proposition}
\newtheorem{lemma}[theorem]{Lemma}
\newtheorem{remark}[theorem]{Remark}

\newtheorem{example}[theorem]{Example}

\renewcommand{\hom}{\textrm{Hom}}
\newcommand{\wt}{\widetilde}

\newcommand{\fix}{\textrm{Fix}}
\begin{document}
\baselineskip=15.5pt
\title[Equivariant cobordism of generalized Dold manifolds]{A note on the equivariant cobordism of generalized Dold manifolds} 
\author[A. Nath]{Avijit Nath} 

\address{Indian Institute of Science Education and Research, Tirupati\\
Rami Reddy
Nagar, Karakambadi Road, Mangalam (P.O.), Tirupati 517507}

\email{avijitnath@imsc.res.in}

\author[P. Sankaran]{Parameswaran Sankaran
}
\address{Chennai Mathematical Institute, H1 SIPCOT IT Park, Siruseri\\ Kelambakkam 603103}
\email{sankaran@cmi.ac.in}
\subjclass[2010]{57R25}
\keywords{Generalized Dold manifolds, elementary abelian $2$-group actions, equivariant cobordism, isotropy representations, flag manifolds}
\thispagestyle{empty}
\date{}
\thanks{Most of this work was done when both the authors were at the Institute of Mathemtaical Sciences, Chennai.  Both the authors were partially supported by a XII Plan Project, Department of Atomic Energy, Government of India.}

\begin{abstract}
Let $(X,J) $ be an almost complex manifold with a (smooth) involution $\sigma:X\to X$ such that $\fix(\sigma)\ne \emptyset$. Assume that 
$\sigma$ is a complex conjugation, i.e, the differential of $\sigma$ anti-commutes with $J$.  
The space $P(m,X):=\mathbb{S}^m\times X/\!\sim$ where $(v,x)\sim (-v,\sigma(x))$ is known as a generalized 
Dold manifold.  
Suppose that a group $G\cong \mathbb Z_2^s$ acts smoothly on $X$ such that $g\circ \sigma
=\sigma\circ g$ for all $g\in G$. Using the action of the diagonal subgroup $D=O(1)^{m+1}\subset O(m+1)$ on the sphere 
$\mathbb S^{m}$ for which there are only finitely many pairs of antipodal points that are stablized by $D$, 
we obtain an action of $\mathcal G=D\times G$ on $\mathbb S^m\times X$, which descends to 
a (smooth) action of $\mathcal G$ on $P(m,X)$.  
When the stationary point set $X^G$ for the $G$ action on $X$ is finite, 
the same also holds for the $\mathcal G$ action on $P(m,X)$. The main result of this note is that the equivariant 
cobordism class $[P(m,X),\mathcal G]$ vanishes if and only if $[X,G]$ vanishes. We illustrate this result in the case 
when $X$ is the complex flag manifold, $\sigma$ is the natural complex conjugation and $G\cong (\mathbb Z_2)^n$ is contained 
in the diagonal subgroup of $U(n)$.

\end{abstract}
\dedicatory{Dedicated to Professor Daciberg Lima Gon\c calves on the occasion of his 70th birthday.  }
\maketitle

\section{Introduction} \label{intro}

Recall that the classical Dold manifold $P(m,n)$ is defined as the orbit space of the $\mathbb Z/2\mathbb Z$-action on $\mathbb S^m\times \mathbb CP^n$ generated by the involution $(v,[z])\mapsto (-v,[\bar{z}]), v\in\mathbb S^m,[z]\in\mathbb CP^n$.
Here $[\bar z]$ denotes $[\bar z_0:\cdots :\bar z_n]$ when $[z]=[z_0:\cdots:z_n]\in \mathbb CP^n.$                
 See \cite{dold}. 

Let $\sigma:X\to X$ be a complex conjugation on an almost complex manifold $(X,J)$, that is, $\sigma$ is an involution such that, for any $x\in X,$ the differential 
$T_x\sigma:T_xX\to T_{\sigma(x)}X$ satisfies the equation $J_{\sigma(x)}\circ T_x\sigma=-T_{\sigma(x)}\sigma \circ J_x$. See \cite[\S 24]{cf}. 
We assume that $\textrm{Fix}(\sigma)\ne \emptyset$. 
The generalized Dold manifold $P(m,X)$ was introduced in \cite{ns} as the quotient of $\mathbb S^m\times X$ 
under the identification $(v,x)\sim (-v,\sigma(x))$.

We obtained a description of its tangent bundle, and, assuming that $H^1(X;\mathbb Z_2)=0$, a formula for the Stiefel-Whitney classes of $P(m,X)$.  
We obtained conditions for the (non) vanishing of the unoriented cobordism class $[P(m,X)]\in \mathfrak N_*$.

Let $G=(\mathbb Z_2)^n$ and consider the class of all smooth compact manifolds (without boundary) admitting 
(smooth) $G$-actions.  
Recall that two manifolds $M_0,M_1$ with $G$-actions 
are $G$-{\it equivariantly cobordant} if there exists a compact manifold-with-boundary $W$ 
admitting a smooth $G$ action on it such that the boundary $\partial W$, with the restricted $G$-action, is equivariantly diffeomorphic to 
the disjoint union $M_0\sqcup M_1$.   The $G$-equivariant 
cobordism class of $M$ with a given $G$-action $\phi$ is denoted $[M,\phi]$ or more briefly $[M,G]$. In this note shall only consider $G$-actions on $M$  with only finite stationary point set 
$M^G=\{x\in M\mid g.x=x~\forall g\in G\}$.   
The set of all $G$-equivariant cobordim classes of compact $G$-mainfolds with finite stationary point sets is a graded 
$\mathbb Z_2$-algebra, denoted $Z_*(G)$ in which 
addition corresponds to taking disjoint union and multiplication to the cartesian products  (with the obvious $G$-actions). One has the forgetful map 
$\epsilon: Z_*(G)\to \mathfrak{N}_*$ to the cobordism algebra sending $[M,\phi]$ to $[M]$.

We consider the action of a subgroup $D$ of $O(1)^{m+1}\subset O(m+1)$ on $\mathbb S^m$ and assume that there are at most 
finitely many pairs of antipodal points in $\mathbb S^m$ that are stable by the action of $D$.  (This is evidently equivalent 
to the requitement that $D$ act on $\mathbb RP^m$ with only finitely many stationary points.) 
For any $G$ action 
on $X$ that commutes with $\sigma$, we obtain an action of $\mathcal G:=D\times G$ on $P(m,X)$ induced by the 
action of $\mathcal G$ on $\mathbb S^m\times X$. If $X^G$ is finite, so is $P(m,X)^\mathcal G$.  

We now state the main result of this paper.   
 
\begin{theorem} \label{main}  Suppose that $G\cong \mathbb Z_2^q$ acts on an almost complex manifold $(X,J)$ 
such that $X^G$ is finite and that the isotropy representation at each $G$-fixed point is complex linear. Suppose that the action of 
$D\subset (O(1))^{m+1}\subset O(m+1)$  on $\mathbb RP^m$ has only finitely many stationary points.   Then $[X,G]= 0\iff [P(m,X),\mathcal G]=0$ where $\mathcal G=D\times G$.  
\end{theorem}

We illustrate, in \S4, the above result when $X$ is a complex flag manifold.     

\section{Equivariant cobordism and the representation ring}
Let $M^d$ be a smooth  compact manifold with a smooth $G\cong (\mathbb Z_2)^q$ action with only finitely many stationary points.  Then $[M,\phi]\in Z_*(G)$ is determined completely by the isotropy representations of $G$ at the stationary points of $M$.  
More precisely, denote by $R_*(G)$ the {\it representation ring} of $G$ whose elements are formal $\mathbb Z_2$-linear combinations of isomorphism 
classes of finite dimensional real representations of $G$. If $U,V$ are two $G$-representations, the product $[U].[V]$ in $R(G)$ is, by definition, the class of $U\oplus V$ (with the diagonal $G$-action). Note that $R(G)$ is graded via the degree of the representation.    We may identify $R_*(G)$ 
with the polynomial algebra over $\mathbb Z_2$  with indeterminates $v_\chi$,  
$\chi\in \hat G:=\textrm{Hom}(G,\mathbb Z_2)$, the group of 
characters of $G$. Under this isomorphism $[V]$ corresponds to the monomial 
$v_{\chi_1}^{m_1}\cdots v_{\chi_r}^{m_r}$ where $V\cong \mathbb R^{m_1}_{\chi_1}\oplus \cdots \oplus \mathbb R^{m_r}_{\chi_r}$. 
Here $\mathbb R^k_\chi$ denotes the direct sum of $k$ copies of the $1$-dimensional representation 
$\mathbb R_\chi\cong \mathbb R$ on which $G$ acts via the character $\chi$.  
Given $(M,\phi)$,  
we obtain an element $\sum_{x\in M^G} [T_xM]\in R(G)$, where $T_xM$ denotes tangent space regarded as the isotropy representation of $G$.    The map 
$\eta_*:Z_*(G)\to R_*(G)$ sending $[M,\phi]$ to $\sum_{x\in 
M^G}[T_xM]$ is a well-defined algebra homomorphism. By a result of Stong \cite{stong}, $\eta_*$ is in fact a monomorphism.    
The finiteness of $M^G$ implies that $T_xM$ does not contain the trivial $G$-
representation.  So the image of $\eta_*$ is contained in the subalgebra of 
$R_*(G)$ generated by the {\it nontrivial}  one-dimensional representations of $G$.

The diagonal subgroup $D_n\cong \mathbb Z_2^n$ of $O(n)\subset U(n)$ acts the complex flag manifold 
$U(n)/U(n_1)\times \cdots\times U(n_r)=\mathbb CG(n_1,\ldots, n_r)$ with finitely many stationary points.  
Here $n=\sum_{1\le j\le r} n_j$.  In fact, the stationary points are precisely the flags $\mathbf L=(L_1,\cdots, L_r)$ where each component $L_j$ is spanned by a subset of the standard basis $e_1,\ldots, e_n$.  Similarly, $D_{m+1}\subset O(m+1)$ acts on 
$\mathbb S^m$.  Although there is no point on the sphere which is stationary, the induced action on 
the real projective space $\mathbb RP^m$ has $m+1$ stationary points, namely,  $[e_1],\ldots, [e_{m+1}]$.  
We obtain an action of $D_{m+n+1}=D_{m+1}\times D_n$ on $\mathbb S^m\times \mathbb C G(n_1,\ldots,n_r)$.  This action 
yields an action of $D_{m+n+1}$ on $P(m,\mathbb C G(n_1,\ldots, n_r))$ with finitely many stationary points, namely, $[e_j,\mathbf L], 1\le j\le m+1,$ 
with $\mathbf L$ as above.  
We shall consider the restricted action of certain 
subgroups of 
$D_{m+n+1}$ on $P(m,\mathbb CG(n_1,\ldots,n_r))$ with finitely many stationary points and obtain results on the (non) vanishing of 
the equivariant cobordism classes of $P(m, \mathbb CG(n_1,\ldots,n_r))$.

\section{Equivariant cobordism of generalized Dold manifolds.}

Let $(X,J)$ be an almost complex manifold and let $\sigma$ be a complex conjugation with non-empty fixed point set.  We denote 
$\fix(\sigma)$ by $X_\mathbb R$.  Suppose that $G\cong \mathbb Z_2^q$ acts smoothly on $X$ such that  
(i) $t\circ\sigma=\sigma \circ t$ for all $t\in G$, 
and, (ii) the stationary point set $X^G$ for the $G$ action is finite. Then $G$ acts on $X_\mathbb R$ with 
$X_\mathbb R^G=X_\mathbb R\cap X^G$. 

We have the following lemma which is a straightforward generalization of \cite[Theorem 24.4]{cf}. 

\begin{lemma} \label{square} 
With the above notations, suppose that $t\circ \sigma=\sigma\circ t$ for all $t\in G$ and that, for each $x\in X^G_\mathbb R$, 
the isotropy representations are $\mathbb C$-linear; equivalently 
$J_x:T_xX\to T_{x}X$ is $G$-equivariant.  Then $[X,G]=[X_\mathbb R,G]^2$ in $Z_*(G)$.    
\end{lemma}  

\begin{proof}
Let $x\in X^G$.  Our hypothesis that $t\circ \sigma=\sigma\circ t$ for all $t\in G$ implies that $T_x\sigma:T_xX\to T_{\sigma(x)}X$ is an 
isomorphism of $G$-modules.  Indeed, for $t\in G$, we have 
$T_x\sigma(t.v)=T_x\sigma (T_xt(v))=T_{\sigma(x)}t(T_x\sigma(v))=t.T_x\sigma(v)$ for all $v\in T_xX$.  
In particular $[T_xX]+[T_{\sigma(x)} X]=0$ in 
$R_*(G)$. 

Suppose that $\sigma(x)=x$, that is, $x\in X_\mathbb R$.  We have $T_xX_\mathbb R \cap J_x(T_xX_\mathbb R)=0$.  To see this, we need only 
observe that $T_xX_\mathbb R$ and $J(T_xX_\mathbb R) $ are the $1$- and $-1$-eigenspaces of $T_x\sigma:T_xX\to T_xX$.  Since $J_x:T_xX_\mathbb R\to J_x(T_x X_\mathbb R)$ is 
$G$-equivariant, $T_xX_\mathbb R$ and $J_x (T_x X_\mathbb R)$ are isomorphic as $G$-modules and so $[T_xX]=[T_xX_\mathbb R][J_x(T_xX_\mathbb R)]=[T_xX_\mathbb R]^2$ 
in $R_*(G)$.  

Now $\eta_*([X,G])=\sum_{x\in X^G} [T_x X]=\sum_{x\in X_\mathbb R^G} [T_xX]+\sum_{x\in X^G\setminus X_\mathbb R}[T_x X]
=\sum_{x\in X^G_\mathbb R}[T_xX_\mathbb R]^2$, since $[T_{\sigma(x)}X]$ cancels out $[T_xX]$ for each 
$x\in X^G\setminus X_\mathbb R$.  
So $\eta_*([X,G])=\sum_{x\in X_\mathbb R^G}[T_xX_\mathbb R]^2=(\sum_{x\in X_\mathbb R}[T_xX_\mathbb R])^2
=\eta_*([X_\mathbb R]^2)$.  Since $\eta_*:Z_*(G) \to R_*(G)$ is a monomorphism, we are done.  
\end{proof}

We remark that when $J$ arises from a complex structure on $X$ and $G$ acts as a group of biholomorphisms, 
$J_x$ is $G$-equivariant for all $x\in X^G$.   

Suppose that for the restricted action of a subgroup $D\subset D_{m+1}\subset O(m+1)$ on the sphere $\mathbb S^m=O(m+1)/O(m)$, the induced 
action on $\mathbb RP^m$ has only finitely many 
stationary points.   These points are precisely $[e_j]\in \mathbb RP^m, 1\le j\le m+1$.

Consider the action of $D\times G$ on 
$\mathbb S^m\times X$. 
We assume that the $t\circ \sigma=\sigma\circ t$ for all $t\in G$ so that 
the action of $\mathcal G:=D\times G$ on $\mathbb S^m\times X$ descends to an action on $P(m,X)$. 
It is readily verified that $P(m,X)^{\mathcal G}$ equals $\{[e_j,x]\mid x\in X^G_\mathbb R, 1\le j\le m+1\}$.  
Evidently, $[P(m,X),\mathcal G]=0$ if $[P(m,X),D_{m+1}\times G]=0$.

Let $x_0\in X_\mathbb R^G$.  Since $[v,x_0]=[-v,\sigma(x_0)]=[-v,x_0]$, we have a well-defined cross-section $s_{x_0}:\mathbb RP^m\to P(m,X)$ defined by $[v]\mapsto [v,x_0]$.  In fact 
$([v],x)\mapsto [v,x]$ is a well-defined imbedding $s: \mathbb RP^m\times X_\mathbb R\to P(m,X).$   We shall denote by $\iota_j: X\hookrightarrow 
P(m,X)$ the fibre-inclusion $x\mapsto [e_j,x], 1\le j\le m+1$.   
We note that, with the trivial $G$ action on $\mathbb RP^m$ understood, the embedding $s_{x_0}$ is $\mathcal G$-equivariant.  In fact, $s$ is $\mathcal G$-
equivariant: if $\gamma=(\alpha,t)\in D\times G$, then, we see that 
$s(\gamma.([v],x))=s([\alpha.v],t.x)=[\alpha. v, t.x]=\gamma.[v,x]=\gamma.s([v],x)$.   
On the other hand, $\iota_j$ is not $\mathcal G$-equivariant since $e_j$ is not $D$-fixed.  However, it turns out that, after twisting 
the action of $\mathcal G$ on $X$, $\iota_j$ becomes $\mathcal G$-equivariant, as we shall now explain.  Let $\chi'_j:D\to \langle \sigma\rangle $ be the homomorphism whose kernel equals the isotropy group $D_j\subset D$ at $e_j\in \mathbb S^m$.  
Define the $\chi'_j$-twisted action of $\mathcal G$ on $X$ by $(\alpha, t)(x):=\chi'_j(\alpha) (t.x)$.  We note that both the $\mathcal G$-
actions agree on 
$X_\mathbb R$ and that the $G$-action obtained from the restriction of the action 
to $p^{-1}([e_j])$ of the twisted $\mathcal G$-action on $P(m,X)$ 
is the same as the that of the original $G$-action on $X$.   We shall denote by $X_j\cong X$ the fibre $p^{-1}([e_j])\subset P(m,X)$. 

We now verify that $\iota_j$ is $\mathcal G$-equivariant with respect to the twisted $\mathcal G$-action on $X$.  
For this purpose, let $\gamma=(\alpha,t)\in \mathcal G, x\in X$.  
Then $\iota_j(\gamma. x)=\iota_j(\chi'_j(\alpha) t .x)=[e_j,\chi_j'(\alpha) t.x]$ while $\gamma.(\iota_j(x))=\gamma.[e_j, x]=[\alpha e_j, t.x]$.  
If $\alpha(e_j)=e_j$, then $\chi_j(\alpha)=1$ and so it follows that $\iota_j(\gamma.x)=\gamma.(\iota_j(x))$.  If $\alpha(e_j)=-e_j,$ then 
$\chi_j(\alpha)=\sigma$ and so $\gamma(\iota_j(x))=[-e_j,t.x]=[e_j,\sigma t.x]=[e_j, \chi_j'(\alpha) t.x]= \iota_j(\gamma.x)$, proving our 
claim.  It follows that the $\iota_{j*}:T_{x_0} X\to T_{[e_j,x_0]} P(m,X)$ is $\mathcal G$-equivariant and  
so, we have a decomposition of the tangent space $T_{[e_j,x_0]}P(m,X)$ into $\mathcal G$-submodules:
\[T_{[e_j,x_0]}P(m,X)=s_{x_0*}(T_{[e_j]} \mathbb RP^m)
\oplus \iota_{j*} (T_{x_0} X)=s_{x_0*}T_{[e_j]}\mathbb RP^m\oplus T_{[e_j,x_0]}X_j. \eqno(1)\] 
Since $x_0\in X_\mathbb R$ and since the twisted $\mathcal G$-action on $X_\mathbb R$ coincides 
with the untwisted action, it follows that $\iota_{j*}T_{x_0}X_\mathbb R\subset T_{[e_j,x_0]}X_j$ is isomorphic as a $\mathcal G$-submodule to $T_{x_0}X_\mathbb R$.   We claim that its $\mathcal G$-complement 
$\iota_{j*}(J_{x_0} T_{x_0}X_\mathbb R)\subset T_{[e_j, x_0]}X_j$ is isomorphic as a $\mathcal G$-module to 
$E_j\otimes T_{x_0}X_\mathbb R$.   Here $E_j=\mathbb R e_j$ the 
one dimensional $D$-representation corresponding to the character $\chi_j$. 
In fact, $\theta: E_j\otimes T_{x_0}X_\mathbb R\to \iota_{j*}J_{x_0}T_{x_0}X_\mathbb R$, defined by $e_j\otimes u\mapsto \iota_{j*}(Ju)$, is 
an isomorphism of $\mathcal G$-modules. 
To see this, let $\gamma=(\alpha,t)\in D\times G$.  We have $\gamma.u=t.u =t_*(u)~\forall u\in T_{x_0}X_\mathbb R$.
Also, $\gamma.e_j=\chi_j(\alpha) e_j=\pm e_j$ where the sign is positive  precisely if $\alpha\in D_j\subset D$.  
Thus $\theta(\gamma (e_j\otimes u))=\theta(\pm e_j\otimes t_*u)=\pm\iota_{j*}J(t_*u)=\pm \iota_{j*}t_*(J u)$ where the sign is positive precisely 
when $\alpha\in D_j\subset D$.   

On the other hand, since $\iota_{j*}$ is $\mathcal G$-equivariant with respect to the 
twisted $\mathcal G$-action on $T_{x_0}X$, 
 $\gamma(\theta(e_j\otimes u))=\gamma (\iota_{j*} Ju)=\iota_{j*}(\gamma .Ju)=\iota_{j*} ((\alpha, t).(Ju))=\iota_{j*}(\chi'_j(\alpha) t_*(Ju))$.  
Note that as $u\in T_{x_0}X_\mathbb R$, $Ju$ is in the $-1$-eigen space of $\sigma_*$. So, from the definition of $\chi_j':D\to \langle\sigma\rangle$, we have  
\[\chi_j'(\alpha)_* .Ju=\left\{ \begin{array}{ll} 
\sigma_*Ju=-Ju & \textrm{~if~} \alpha\notin D_j\\
Ju& \textrm{~if~} \alpha\in D_j.\\
\end{array}
\right.  \eqno(2)\]
Consequently, 
 $\gamma\theta(e_j\otimes u)=\pm \iota_{j*}( t_*(Ju))$ where the sign is positive precisely when $\alpha\in D_j\subset D$. Hence $\theta$ is a $\mathcal G$-isomorphism. Therefore $\theta$ is $\mathcal G$-isomorphism.

We summarise the above discussion in the following. 

\begin{proposition} \label{tgt}  Suppose that $t_*J_{x_0}=J_{x_0}t_*$ for all $t\in G, ~x_0\in X^G_\mathbb R$.  
With the above notations,   
we have, for $1\le j\le m+1$ and $x_0\in X_\mathbb R^G$,  
an isomorphism of $\mathcal G$-modules:
 \[T_{[e_j,x_0]}P(m,X)\cong 
T_{[e_j]}\mathbb RP^m\oplus T_{x_0}X_\mathbb R\oplus (E_j\otimes T_{x_0}X_\mathbb R).\eqno(3) \] \hfill  $\Box$
\end{proposition}

We apply the above proposition to identify the image of $[P(m,X),\mathcal G]$ in $R(\mathcal G)$ under $\eta_*$.

 Since $D$ is possibly a proper subgroup of $D_{m+1}$, the characters $\chi_j$ are not necessarily 
linearly independent.  However the $D$-representations $E_j, 1\le j\le m+1,$ are pairwise non-isomorphic; equivalently the characters 
$\chi_j, 1\le j\le m+1,$ are pairwise distinct. Indeed, suppose that for some $i\ne j$, $E_i\cong E_j$, 
then $D_i=D_j$ and so, for any $a,b\in \mathbb R$, we have
$\chi.(ae_i+be_j)=\pm(ae_i+be_j)~\forall \chi\in D$.  This implies that, for any $a,b$ not both zero, the point  
$\mathbb R(ae_i+be_j)\in \mathbb RP^m$ is $D$-fixed.  
This contradicts our hypothesis that $\mathbb RP^m$ has only finitely many stationary points for the $D$-action.  
We shall write $\chi_j$ to also denote the isomorphism class $[E_j]\in R(D)$ of the irreducible 
representation of $D$. 

\noindent
{\it Notation.} Let $k$ be a positive integer.
We shall denote by $[k]$ the set $\{1,\cdots, k\}$.

Fix a basis $t_r, 1\le r\le q,$ for $G$. Then the character group $\hat G$ consists of elements $y_\alpha, \alpha\subset [q],$ where $y_\alpha(t_j)= -1$ if and only if 
$j\in \alpha$.  We abbreviate $y_{\{j\}}, y_{\{i,j\}}$ to $y_j, y_{i,j}, 1\le i, j\le q, i\ne j,$ respectively and denote $y_\emptyset $ by $y_0$.  
The $y_\alpha$ generate $R(G)$ as a $\mathbb Z_2$-algebra.  
An entirely analogous notation is used for $R(D)$.  
The elements $\chi_j, y_r, 1\le j\le m+1, 
1\le r\le q,$ form a basis for $\hat {\mathcal G}$ and we have $R(\mathcal G)=R(D)\otimes R(G)$.

The representation ring has an additional structure arising from tensor product of representations.  If $A=E_\alpha, 
B=E_\beta$ are $1$-dimensional representations of $R(D)$ corresponding to characters $\chi_\alpha,\chi_\beta$, then 
$A\otimes B$ has character $\chi_{\alpha\Delta\beta}$ where $\alpha\Delta \beta$ stands for the symmetric 
difference $\alpha\cup \beta\setminus \alpha\cap \beta$.  Note that the group operation in $\hat D$ is given by 
symmetric difference: $\chi_\alpha.\chi_\beta=\chi_{\alpha\Delta\beta}$. But we will avoid using the notation 
$\chi_\alpha.\chi_\beta\in \hat D$ since the product has been given a different meaning in $R(D)$; instead we shall denote 
this character by $\chi_\alpha\otimes \chi_\beta$.   
In the more general case where $[A]=\chi_{\alpha_1}\cdots\chi_{\alpha_r}, [B]=\chi_{\beta_1}\cdots\chi_{\beta_s}\in R(D)$, 
we have $[A\otimes B]=\prod_{1\le i\le r,1\le j\le s}(\chi_{\alpha_i}\otimes \chi_{\beta_j})$.   

We will need to consider tensor product representations of $\mathcal G$ of the form $E\otimes V$ where $E, V$ are 
representations of $D$ and $G$ respectively with $E$ being $1$-dimensional.  

It will be convenient to  identify 
$\chi_\alpha$ with $\chi_\alpha\otimes y_0$ and $ y_\beta$ with 
$\chi_0\otimes y_\beta$  for $\alpha\subset 
[m+1], \beta\subset [ q]$. 
Let $[E]=\chi\in R(D)$ and 
suppose that $[V]=y_{\beta_1}\cdots y_{\beta_d}\in R(G)$.  Then $[E\otimes V]
=(\chi\otimes y_{\beta_1})\cdots (\chi\otimes y_{\beta_d})\in R(\mathcal G)$.

As is well-known, 
$T_{[e_j]}\mathbb RP^m=E_j\otimes E_j^\perp$; here $E_j^\perp=\oplus_{1\le i\le m+1, i\ne j} E_i \subset \mathbb R^{m+1}$, the orthogonal complement to $E_j$ in $\mathbb R^{m+1}$.  
(See \cite{ms}.)  
This is in fact an isomorphism of $\mathcal G$-modules where the $\mathcal G$ action is via the projection $\mathcal G\to D$.  So 
$[T_{[e_j]}\mathbb RP^m]=\prod_{1\le i\le m+1, i\ne j}\chi_{i}\otimes \chi_j\in R(\mathcal G).$    
 
The following proposition is now immediate from Proposition \ref{tgt}.

We shall denote by $f(y_\beta)\in R(G)$ a polynomial in the variables $y_\beta,\beta\subset [q]$.

\begin{proposition}\label{gendold}  Suppose that $G\cong \mathbb Z_2^q$ acts on $(X,J)$ with finitely many stationary 
points such that that the isotropy representation is $\mathbb C$-linear at each point of $X_\mathbb R^G$.  Suppose that 
the action 
of $D\subset D_{m+1}$ on $\mathbb RP^m$ has only finitely many stationary points.  
With the above notation, 
set $f_p(y_\beta):=[T_pX_\mathbb R]\in R(G), p\in X_\mathbb R^G$.  Then 
\[\eta([P(m,X),\mathcal G])
=\sum_{p\in X_\mathbb R^G}\sum_{1\le j\le m+1} \prod_{1\le i\le m+1, i\ne j} \chi_{i}\otimes \chi_j. f_{p}(y_\beta)f_p(\chi_j\otimes y_\beta) \eqno(4) \]
where $\mathcal G=D\times G$. \hfill $\Box$
\end{proposition}

We now turn to the proof of Theorem \ref{main}.

\noindent 
{\it Proof of Theorem \ref{main}.} It suffices to show that $\eta([X,G])=0\iff \eta([P(m,X),\mathcal G])= 0$.

First suppose that $\eta([X,G])=0$.  Then  we have $[X_\mathbb R,G]=0$ in view of Lemma \ref{square} and 
the fact that $Z_*(G)\subset R(G)$ has no non-zero nilpotent elements.  So the $G$-stationary points of $X_\mathbb R$ occur in pairs $p,q$, $p\ne q$, such that 
$T_pX\cong T_qX$ as $G$-modules.  It follows that $E_j\otimes T_pX\cong E_j\otimes T_qX$ as $\mathcal G$-modules.  
Hence $T_{[e_j,p]}P(m,X)\cong T_{[e_j,q]} P(m,X)$ for every $j\in [m+1]$ 
by Proposition \ref{tgt}.  It follows that $\eta([P(m,X),\mathcal G])$ vanishes.  

For the converse part, assume that  $[P(m,X),\mathcal G]=0$.  This means that the there is a fixed point free 
bijective correspondence $(i,p)\leftrightarrow (j,q)$ on $P(m,X)^\mathcal G$ 
such that $T_{[e_i,p]} P(m,X)\cong T_{[e_j,q]}P(m,X)$ as $\mathcal G$-modules.
That is, 
\[T_{[e_i]}\mathbb RP^m\oplus T_pX_\mathbb R\oplus E_i\otimes T_pX_\mathbb R\cong T_{[e_j]}\mathbb RP^m
\oplus T_{q}X_\mathbb R\oplus E_j\otimes T_qX_\mathbb R \eqno(5)\] 
as $\mathcal G=D\times G$-module, by Proposition \ref{tgt}. 

Restricting to $D=D\times 1\subset \mathcal G,$ we obtain
 \[T_{[e_i]}\mathbb RP^m\oplus \mathbb R^n\oplus 
E_i\otimes \mathbb R^n= T_{[e_j]}\mathbb RP^m\oplus \mathbb R^n\oplus E_j\otimes \mathbb R^n,\eqno(6) \] 
where $\mathbb R^n$ is the 
trivial $D$-representation of degree $n$.  

Since  $T_{[e_i]}\mathbb RP^m\oplus \mathbb R=(\oplus_{i\ne k}E_k\otimes E_i)\oplus E_i\otimes E_i=V\otimes E_i$ where 
$V=\oplus_{1\le k\le m+1} E_k$, 
cancelling $\mathbb R^{n-1} $ on both sides and using $E_i\otimes E_i=\mathbb R$ 
we obtain from (6) that $(V\otimes E_i)\oplus (E_i\otimes \mathbb R^n)=(V\otimes E_j)\oplus (E_j\otimes\mathbb R^n)$.  That is,
 \[(V\oplus \mathbb R^n)\otimes E_i=(V\oplus \mathbb R^n)\otimes E_j \eqno(7)\] as $D$-modules.  
Let $k=\dim_\mathbb R Hom_D(\mathbb R,V)=\dim_\mathbb R\hom_D(E_i, V\otimes E_i)$, the multiplicity of the trivial ($1$-dimensional) submodule in $V$ and let $l=
\dim_\mathbb R \hom(E_{i}\otimes E_j, V)=\dim_\mathbb R\hom_D(E_i,V\otimes E_j)$.  
As has been observed already, multiplicity of {\it any} one-dimensional representation occurring in $V$ is at most $1$ and so, in particular, 
$k,l\le 1$.  

Suppose that $i\ne j$, so that $E_i, E_j\subset V$ are not isomorphic.  
Comparing the multiplicities of  
of $E_i$ on both sides of the above isomorphism we get $k+n=l$.  
This is a contradiction since $n\ge 2$ and $l\le 1.$ 
So we must have $i=j$. 

Since $(i,p)\ne (j,q)=(i,q),$  
we have $p\ne q$. 
Thus, fixing $i=1$,  we obtain a fixed point free bijection $p\leftrightarrow q$ of $X_\mathbb R^G$ where $q$ is such that 
$(1,p)\leftrightarrow (1,q)$.  
Now restricting the $\mathcal G$-isomorphism to $G$, we see that $T_pX_\mathbb R\cong T_q X_\mathbb R$.  
It follows that $[X_\mathbb R,G]=0$ and so $[X,G]=0$ by Lemma \ref{square}.  \hfill $\Box$

\begin{remark}\label{vanish} {\em 
(i) Suppose that there exists an involution $\theta:P(m,X)\to P(m,X)$ which commutes with each $\gamma\in \mathcal G$.
Let $\wt{\mathcal G}$ be the group $\mathcal G\cup \{\theta\circ \gamma\mid \gamma\in \mathcal G\}.$  If $P(m,X)^{\wt{\mathcal G}}=\emptyset$, 
then $[P(m,X),\wt{\mathcal G}]=0$ in $Z_*(\wt{\mathcal G})$ and so $[P(m,X),\mathcal G]=0$ in $Z_*(G)$.  

(ii) More generally, if $H\cong \mathbb Z_2^p$ acts smoothly on $P(m,X)$ such that $\theta\circ \gamma=\gamma\circ \theta$ 
for all $\theta\in H,\gamma\in \mathcal G$.  Then $\wt{\mathcal G}=H\times\mathcal G$ acts on $P(m,X)$.  If $P(m,X)^{\wt {\mathcal G}}
=\emptyset$, then $[P(m,X),\mathcal G]=0$ in $Z_*(\mathcal G)$ since $[P(m,X);\wt{\mathcal G}]=0$. 

(iii) Let $V$ be an $(m+1)$-dimensional $D$-representation where the multiplicity of each one-dimensional representation 
occurring in $V$ equals $1$ so that the induced $D$-action on $\mathbb RP^m$ has only finitely many stationary points.  
Using the observation that $T_{[e_i]} \mathbb RP^m\oplus \mathbb R\cong V\otimes E_i$ as in the course of the above proof 
we see that, given distinct positive integers $i,j\le m+1$ the $D$-representations  $T_{[e_i]}\mathbb RP^m, T_{[e_j]}\mathbb RP^m$ are isomorphic if and only if $V\otimes E_j\cong V\otimes E_i$, 
if and only if, for each $k\le m+1,$ there exists an $l\le m+1$ such that $E_k\otimes E_ i\cong E_l\otimes E_j$;
equivalently $\chi_k\otimes \chi_i=\chi_l\otimes \chi_j$. Therefore, $[\mathbb RP^m,D]\ne 0$ if and 
only if, for some $i,k \le m+1, i\ne k,$ we have $\chi_k\otimes \chi_i\ne \chi_l\otimes \chi_j$ for all  $l$ and any $j\le m+1, j\notin\{i,k\}$.

}
\end{remark}

\section{Group actions on $P(m;n_1,\ldots,n_r)$ with finite stationary point sets}

Let $G=\mathbb Z_2^q$ with standard $\mathbb Z_2$-basis $t_j, 1\le j\le q$ and $y_j,1\le j\le q, $ denote the basis for the dual $\hat G=\textrm{Hom}
(G,\mathbb C^\times)\cong \mathbb Z^q_2$.  We shall use multiplicative notation for group operation in $G$ and $\hat G$.   The elements of $G$ and 
$\hat G$ are labeled by subsets of $[q]:=\{1,\ldots, q\}$ where $t_\alpha=\prod_{j\in \alpha}t_j, y_\alpha(t_j)=-1$ if and only if $j\in \alpha$, 
$\alpha\subset [ q]$.
It follows that $y_\alpha(t_\beta)=-1$ if and only if the cardinality $\#(\alpha\cap \beta)$ is odd.

 Let  $V_\mathbb R=\mathbb RG$ denote the regular representation of $G$ and $V=V_\mathbb R\otimes \mathbb C$ its complexification.  Then $V\cong \mathbb C^{2^q}$ decomposes into $1$-dimensional 
complex representations as follows: $V =\oplus_{y\in\hat G} \mathbb C_y$; the $G$-action on $\mathbb C_y$ is via the 
character $y$, that is, $t.z=y(t)z ~\forall t\in G$.  

Let $n\le 2^q$ and let $n_1,\ldots, n_r$ be positive integers such that $n=\sum_{1\le i\le r} n_i.$    
Let $U\subset V$ be an $n$-dimensional complex $G$-submodule. 
The $G$-action on $U\cong \mathbb C^{n}$ yields an action $\phi$ on the complex flag manifold $X:=
\mathbb CG(n_1,\ldots,n_r)\cong U(n)/(U(n_1)\times \cdots\times U(n_r))$ which is identified with the space of flags 
${\bf L}:=(L_1,\cdots, L_r)$ where the $L_j$ are $\mathbb C$-vector subspaces of $U$ such that $\dim L_j=n_j, 1\le j\le r,$ and $ L_i \perp L_j $ if $i\ne j$.   

The tangent bundle 
of $ X$ has the following description as a complex vector bundle: Let $\xi_j$ denote the complex vector bundle 
over $X$ whose fibre over a flag ${\bf L}\in X$ is the vector space $L_j, 1\le j\le r$.  We shall denote by $\bar{\xi}_j$ complex conjugate of $\xi_j$.  The 
fibre of $\bar{\xi}_j$ over $\mathbf L$ is the vector space $\bar{L}_j\subset \mathbb C^n$.  Then, by \cite{lam},
\[\tau X=\bigoplus_{1\le i<j\le r}\bar{\xi}_i\otimes_\mathbb C\xi_j. \eqno(8)\]

It is easily seen that there 
are only finitely many stationary points for this $G$-action using the fact that multiplicity of any irreducible complex representation 
occurring in $U$ is at most $1$.  In fact the stationary points are precisely the flags ${\bf L}$ in which each $L_j$ is a $G$-submodule of $U$.  
Cf. \cite[Lemma 2.1]{gp}.   If $S\subset \hat G$ is the set of all characters $y$ such that 
$\mathbb C_y$ occurs in $U$, we shall denote this action on $\mathbb CG(n_1,\ldots,n_r)$ by $\phi_S$.  

Let $U_\mathbb R=U\cap V_\mathbb R\cong \mathbb R^n$.  
The action of $G$ on $U$ restricts to an action on $U_\mathbb R$ and so we obtain a $G$-action on the real flag manifold $\mathbb RG(n_1,\ldots, n_r)=O(n)/(O(n_1)\times \cdots\times O(n_r))$, 
which we denote by $\phi_S^\mathbb R$.  

Since the $G$ action on $V$ commutes complex conjugation, denoted $\sigma$  (defined as $\sum z_\alpha t_\alpha\mapsto \sum \bar z_\alpha t_\alpha$),  it follows that  
the $G$-action on $\mathbb CG(n_1,\ldots, n_r)$ commutes with the complex conjugation on it, again denote $\sigma$.   Therefore $G$ acts on $\textrm{Fix}(\sigma)=X_\mathbb R$.  
Note that $X_\mathbb R$ is naturally identified with 
the real flag manifold $\mathbb RG(n_1,\ldots, n_r)$.   The identification is obtained as $\textrm {Fix}(\sigma)\ni 
(L_1,\ldots, L_r)\mapsto (L_1\cap U_\mathbb R,\ldots, L_r\cap U_\mathbb R)\in \mathbb RG(n_1,\ldots, n_r)$.   Under this identification,  
the restricted action of $\phi_S$ on $\textrm{Fix}(\sigma)$ corresponds to $\phi_S^\mathbb R$.   The stationary points 
for the $G$-action on $\mathbb CG(n_1,\ldots, n_r)$ are all contained in $\mathbb RG(n_1,\ldots, n_r)$.  Indeed, a flag $(L_1,\ldots,L_r)\in 
\mathbb CG(n_1,\ldots,n_r)$ is fixed by every element of $G$ if and only if each $L_j\cap U_\mathbb R$ is a $G$-submodule of $U_\mathbb R$. 
We shall henceforth identify $X_\mathbb R$ with the real flag manifold. 

In view of Lemma 3.1 and Theorem \ref{main} we see that $[P(m;n_1,\ldots, n_r),\mathcal G]$ is zero if and only if $[X_\mathbb R,\phi_S^\mathbb R]=0$ (where $X_\mathbb R=\mathbb RG(n_1,\ldots, n_r)$).
We shall now give examples of $(X_\mathbb R,\phi^\mathbb R_S)$ that are equivariantly null-cobordant.

\begin{example}{\em 
(i). If the Euler characteristic $\chi(X_\mathbb R)$ is odd, then $[X_\mathbb R]\ne 0$ in $\mathfrak N_*$, which implies that $[X_\mathbb R,\phi^\mathbb R_S]\ne 0$ 
for any $S$.  As is well-known $\chi(X_\mathbb R)=n!/(n_1!\cdots n_r!)$. It is a classical result that the multinomial coefficient $n!/(n_1!\cdots n_r!)$ is odd if and only if $\alpha(n)=\sum_{1\le j\le r} \alpha(n_j)$ where $\alpha(k)$ is the number of $1$s in the 
dyadic digits of $k$. 

(ii). Suppose that $n_i=n_j$ for some $i< j\le r$.  Then we have a fixed point free $\mathbb Z_2$-action $\theta$ given by the involution $\theta :X_\mathbb R\to X_\mathbb R$ that interchanges 
the $i$-th and the $j$-th components of each flag in $X_\mathbb R$.   It is trivial to verify that 
$\theta$ commutes with the $G$-action $\phi_S$ on $X_\mathbb R$.   This leads to a fixed point free action of $\wt G:=G\times \mathbb Z/2\mathbb Z$ action on $X_\mathbb R$. So $[X_\mathbb R,
\wt{G}]=0$.  This implies that $[X_\mathbb R,G]=0$.

(iii) Let $n=2^q-2$.    Let $S\subset \hat G$ be any subset which does not contain the trivial representation and $\# S=n$.  It was shown in \cite[Example 2.3]{gp} that if $k<n$ odd, then 
$[\mathbb RG(k,n-k),\phi_S]=0$.  We consider here the 
more general case of a flag manifold $X_\mathbb R=\mathbb RG(n_1,\ldots, n_r), r\ge 3,$ where we assume that $n_1$ is odd.   We claim that, as in the case of the Grassmannian, 
$[X_\mathbb R,\phi_S]=0$. Although this is a routine generalisation of the case of Grassmann manifolds, we give most of the details.\\
{\it Proof of Claim:} In view of our hypotheses, $S$ leaves out exactly one nonempty subset $[ q]$ which we shall denote by $\gamma$.  
As in \cite{gp}, we shall denote the symmetric difference of two sets $\alpha, \beta$ by $\alpha+\beta$ and exploit the Boolean algebra structure on the power set of $[q]$.  
Note that 
$\alpha\ne \alpha+\gamma$ for any $\alpha$ since $\gamma\ne \emptyset$.   Also $\alpha\mapsto \alpha+\gamma$ is an involutive  bijection $S\to S$ since $\emptyset,\gamma$ are not in $S$.  
If $A\subset S$ then we shall denote by $A^\gamma$ its image $ \{\alpha+\gamma\mid \alpha\in A\}\subset S$ under this bijection.  We see that, if $\#A$ is odd, 
then $A\ne A^\gamma$. 
If $E\subset U_\mathbb R$ is spanned by standard basis vectors $e_{\alpha_1},\ldots e_{\alpha_k}$, 
we shall denote by $E^\gamma$ the span of $e_{\gamma+\alpha_j},1\le j\le k$.  Note that $E^\gamma\ne E$ if $k=\dim E$ is odd.  
Suppose that $\mathbf{E}=(E_1,\cdots, E_r)\in X_\mathbb R$ is a stationary  point of $\phi_S^\mathbb R$.  Then so is $\mathbf{E}^\gamma:=
(E^\gamma_1,\ldots,E^\gamma_r)$.  Moreover, since $n_1=\dim E_1$ is odd, we see that $\mathbf E^\gamma\ne \mathbf E$.  
So $\mathbf E\mapsto \mathbf{E}^\gamma$ is a fixed point free involution on the set of stationary points of $X_\mathbb R$.  
We shall show that $T_{\bf E} X_\mathbb R\cong T_{{\bf E}^\gamma}X_\mathbb R$ as $G$-representations. The claim follows from this since $\eta_*:Z_*(G)\to R(G)$ is a monomorphism. 
As in the case of complex flag manifolds, the tangent bundle $\tau X_\mathbb R$ has the following description, due to Lam \cite{lam}. 
\[ \tau X_\mathbb R\cong \bigoplus_{1\le i<j\le r}\xi^\mathbb R_i\otimes_\mathbb R \xi^\mathbb R_j \eqno(9) \]
where $\xi^\mathbb R_i$ is the canonical $n_i$-plane bundle over $X_\mathbb R$ whose fibre over $\mathbf L=(L_1,\ldots, L_r)$ is the real vector space $L_i$.  Thus 
$T_{\mathbf E}X_\mathbb R=\oplus_{1\le i<j\le r}E_i\otimes E_j$.   Let $\alpha(i) \subset S$ be defined by the requirement that $E_i$ is the span of $e_p, p\in \alpha(i), 1\le i\le r$.  
Then $E_i\otimes E_j$ is isomorphic, as a $G$-representation, to $\oplus_{p\in \alpha(i)+\alpha(j)}E_{p}$.  Therefore, using $\alpha(i)+\gamma+\alpha(j)+\gamma=\alpha_i+\alpha_j$,  
$E_i\otimes E_j\cong E^\gamma_i\otimes E_j^\gamma$ as $G$-representations 
for $1\le i<j\le r$. It follows that $T_{\mathbf E}X_\mathbb R\cong T_{\mathbf E^\gamma} X_\mathbb R$ as $G$-representations.
This completes the proof.
}
\end{example}

\end{document}